\documentclass{article}
\pdfoutput=1

\usepackage{cmap}
\usepackage{amssymb,amsmath}

\usepackage[T2A]{fontenc}
\usepackage{amsthm}
\usepackage{amsmath, amssymb}
\usepackage{graphicx}
\usepackage{xcolor}
\usepackage[pdftex,unicode=true,colorlinks=true]{hyperref}
\hypersetup{colorlinks,
	pdftitle={Sklyar-Ignatovich-Paper},
	pdfauthor={Sklyar,Ignatovich},
	allcolors=[RGB]{000 000 000}}

\numberwithin{equation}{section}

\theoremstyle{plain}
\newtheorem{thm}{Theorem}

\theoremstyle{definition}

\theoremstyle{remark}
\newtheorem{rem}{Remark}

\newcommand{\RR}{{\mathcal R}}
\newcommand{\R}{{\mathbb R}}

\newcommand{\rank}{{\mathrm{rank}}}

\begin{document}
\title{Solving the time-optimal control problem for nonlinear non-autonomous linearizable systems
\thanks{This work was financially supported by Polish National Science Centre grant no. 2017/25/B/ST1/01892.}}

\author{Katerina~V.~Sklyar\thanks{Institute of Mathematics, University of Szczecin, Wielkopolska str. 15, Szczecin 70-451, Poland (e-mail: jekatierina.sklyar@usz.edu.pl),} {\ }\thanks{Department of Automatics and Robotics, West  Pomeranian University of Technology in Szczecin,  Sikorsky str. 37, Szczecin 70-313, Poland.}, 
\
Svetlana~Yu.~Ignatovich\thanks{Department of Applied Mathematics, V.N. Karazin Kharkiv National University, Svobody sqr. 4, Kharkiv 61022, Ukraine (e-mail: ignatovich@ukr.net, s.ignatovich@karazin.ua).} 
}

\date{}

\maketitle

\begin{abstract} We present the conditions under which the time-optimal control problem for a nonlinear non-autonomous linearizable system can be solved by the method of successive approximations, at each step of which a power Markov moment min-problem is solved. The proposed method can be efficiently implemented by use of symbolic and numerical calculations.   


\end{abstract}

\section{Introduction}

The linearizability problem is an important issue for nonlinear control theory. For nonlinear systems that turn out to be linearizable, well-developed methods of linear control theory can be applied. In this paper, we propose a method for solving the time-optimal control problem for non-autonomous linearizable systems with a single input. 

Let us consider a control system 
\begin{equation}\label{system_nonlin_gen}
	\dot x=f(t,x,u), \ \ x\in\R^n, \ u\in\R^1,
\end{equation}
and suppose that $f\in C^1((-\delta,\delta)\times U(0)\times \R)$ where $\delta>0$ and $U(0)\subset\R^{n}$ is a neighborhood  of the origin. We say that the system (\ref{system_nonlin_gen}) is \textit{locally analytically linearizable at the origin} if there exists a neighborhood $(-\delta',\delta')\times U'(0)\subset(-\delta,\delta)\times U(0)$ and a local change of variables $z=F(t,x)\in C^2((-\delta',\delta')\times U'(0))$ such that the system in the new variables is linear, i.e., takes the form
\begin{equation}\label{system_lin_y}
	\dot z=A(t)z+b(t)u,
\end{equation}
where components of $A(t)$ and $b(t)$ are real analytic in $(-\delta',\delta')$. Here and below under ``local change of variables'' we mean a map $F(t,x)$ which takes the origin to itself and is locally invertible w.r.t.\ $x$, i.e., 
\begin{displaymath}
	F(t,0)\equiv 0, \ \ \det F_x(t,x)\ne0, \ \ (t,x)\in(-\delta',\delta')\times  U'(0),
\end{displaymath}
where the sub-index means the derivative in $x$, i.e., $F_x(t,x)=\frac{\partial F(t,x)}{\partial x}$. Clearly, this is true (maybe in a smaller neighborhood) if $\det F_x(0,0)\ne0$.  

The linearizability property can be used for solving the local controllability problem for the system (\ref{system_nonlin_gen}): for two given points $x(0)=x^0$, $x(T)=x^1$, find $z^0=F(0,x^0)$ and $z^1=F(T,x^1)$ and then find a control $u(t)$ which steers the linear system (\ref{system_lin_y}) from $z^0$ to $z^1$ in the time $T$; then this control steers the system (\ref{system_nonlin_gen}) from $x^0$ to $x^1$.  In this paper we propose a method for solving the time-optimal control problem under the constraint $|u(t)|\le1$. 

First of all, an efficient method of solving the linear time-optimal control problem should be involved. In Subsection~\ref{subsec2_1} we recall known results related to systems of the form (\ref{system_lin_y}), where $A(t)$ and $b(t)$ are real analytic in a neighborhood of zero. For start points from a neighborhood of the origin, an optimal control equals $\pm1$ and has no more than $n-1$ switchings. The direct substitution of such a control leads to a system of $n$ nonlinear equations with $n$ unknowns (switching times and the optimal time). However, under some conditions, the optimal control can be found by the method of successive approximations, at each step of which \textit{a power Markov moment min-problem with gaps} is solved. The power Markov moment problem was originated in \cite{M}, a deep discussion can be found in \cite{KN}. The statement of the Markov moment min-problem and its application to the time-optimal control problem was proposed in \cite{KS}, \cite{KS1}; in many cases it admits an effective solution.

Then, in Subsection~\ref{subsec2_2}, we recall some recent results on \textit{linearizability conditions for non-autonomous systems} proposed in \cite{K}. Additionally to linearizability conditions known since \cite{Krener}, \cite{JR} and generalized to systems of the class $C^1$ in \cite{SCL_2005}, \cite{CMA_2014}, in the non-autonomous case some new conditions arise, see \cite{K1} and \cite{K2} for further discussion. 

Finally, in Section~\ref{sec3} we combine the known results mentioned above and formulate the main result of the paper (Theorem~\ref{th3}), which gives a method for solving the time-optimal control problem for non-autonomous linearizable systems. This method can be effectively used for numerical application; we demonstrate it by an illustrative example in Section~\ref{sec4}.

\section{Background}

\subsection{Solving the time-optimal control problem for linear non-autonomous system}\label{subsec2_1} 

Consider a control system of the form
\begin{equation}\label{system_lin}
	\dot x=A(t)x+b(t)u,
\end{equation}
where the matrix $A(t)$ and the vector $b(t)$ are real analytic in a neighborhood of zero. Denote
\begin{displaymath}
	\textstyle L_i=\frac{1}{i!}\left(-A(t)+\frac{d}{dt}\right)^ib(t)|_{t=0}, \ \ i\ge0,
\end{displaymath}
and suppose 
\begin{equation}\label{cond1}
	\rank\{L_i\}_{i=0}^\infty=n. 
\end{equation}
Let $k_1,\ldots,k_n$ be the indices of the first $n$ linearly independent vectors from the sequence $\{L_i\}_{i=0}^\infty$. Denote 
\begin{displaymath}
	L=(-L_{k_1},\ldots,-L_{k_n}).
\end{displaymath}
The condition (\ref{cond1}) implies that the system (\ref{system_lin}) is locally controllable in a neighborhood of the origin. Let $u(t)$ be a control that steers a point $x^0$ from this neighborhood to the origin, i.e., 
\begin{displaymath}
	\dot x=A(t)x+b(t)u(t), \ \ x(0)=x^0, \ \ x(\theta)=0.
\end{displaymath} 
Then 
\begin{equation}\label{oper}
	x^0=-\int_0^\theta\Phi^{-1}(t)b(t)u(t)dt=-\sum_{j=0}^\infty L_j\int_0^\theta t^ju(t)dt,
\end{equation}
where the matrix $\Phi(t)$ is such that $\dot \Phi(t)=A(t)\Phi(t)$, $\Phi(0)=I$. This means that the right hand side of (\ref{oper}) is a series of \textit{power moments} of the function $u(t)$ with vector coefficients $L_j$. Equality (\ref{oper}) implies
\begin{equation}\label{oper1}
	(L^{-1}x^0)_i=\int_0^\theta t^{k_i}u(t)dt+\sum_{j=k_i+1}^\infty\alpha_{ji}\int_0^\theta t^{j}u(t)dt, \ \ i=1,\ldots,n,
\end{equation}  
where $\alpha_{ji}$ are components of the vector $-L^{-1}L_j$. Below we suppose $|u(t)|\le1$, then $|\int_0^\theta t^ju(t)dt|\le \frac{1}{j+1}\theta^{j+1}$. This means that locally, for small $\theta$, the first term in the right hand side of (\ref{oper1}) is a ``leading'' one. Having this in mind, we consider the \textit{power Markov moment min-problem with gaps} \cite{KS}, \cite{KS1}, \cite{KS_gap}
\begin{equation}\label{mom_gaps}
	y_i=\int_0^\theta t^{k_i}u(t)dt, \ \ i=1,\ldots,n, \ \ |u(t)|\le1, \ \ \theta\to\min.
\end{equation} 
As was shown in \cite{JMAA}, the solution $(\widehat\theta(x^0),\widehat u(t;x^0))$ of the time-optimal control problem 
\begin{equation}\label{tocp}
	\dot x=A(t)x+b(t)u, \ \ x(0)=x^0, \ x(\theta)=0, \ \ |u(t)|\le1, \ \ \theta\to\min
\end{equation}
and the solution $(\theta(y),u(t;y))$ of the power Markov moment min-problem (\ref{mom_gaps}) for ${y=L^{-1}x^0}$ are \textit{equivalent at the origin}, i.e.,
\begin{displaymath}
	\frac{\widehat\theta(x^0)}{\theta(L^{-1}x^0)}\to1, \ \ \frac{1}{\theta}\int_0^\theta|\widehat u(t;x^0)-u(t;L^{-1}x^0)|dt\to0 \ \ \mbox{ as } \ \ x^0\to0.
\end{displaymath} 

Under some additional conditions this result can be strengthened, namely, a fixed-point iteration can be used for finding the solution \cite{KS1}. In \cite{JMAA}, the following theorem was proved.
\begin{thm}
Consider the system (\ref{system_lin}) where $A(t)$ and $b(t)$ are real analytic in a neighborhood of zero and assume the condition (\ref{cond1}) holds. Suppose also that 
\begin{equation}\label{cond2}
	L_i=0 \ \ \mbox{for all} \ \ i<k_n \ \mbox{such that} \ i\ne k_j, \ j=1,\ldots,n-1.
\end{equation}
Then there exists a neighborhood $U(0)$ of the origin such for any $x^0\in U(0)$ the solution $(\widehat\theta(x^0),\widehat u(t;x^0))$ of the time-optimal control problem (\ref{tocp}) can be found as
\begin{equation}\label{s_app1}
	\widehat\theta(x^0)=\lim_{r\to\infty}\theta(y^r), \ \ \widehat u(t;x^0)=\lim_{r\to\infty}u(t;y^r),
\end{equation} 
where $(\theta(y),u(t;y))$ denotes the solution of the Markov moment min-problem (\ref{mom_gaps}) and the sequence $\{y^r\}_{r=0}^\infty$ is defined recursively as  
\begin{displaymath}
	y^0=L^{-1}x^0, \ \ y^{r+1}=L^{-1}\left(x^0+\int_0^{\theta(y_r)}\Phi^{-1}(t)b(t)u(t;y^r)dt\right)+y^r, \ \ r\ge0.
\end{displaymath} 
\end{thm} 

This result follows from the fact that under the condition (\ref{cond2}) the map 
\begin{displaymath}
	y\mapsto L^{-1}\left(x^0+\sum_{j\ne k_i}L_j\int_0^{\theta(y)}t^ju(t;y)dt\right)
\end{displaymath}
is a contraction in a neighborhood of the origin; if $\bar y$ is its fixed point, then $(\widehat\theta(x^0),\widehat u(t;x^0))=(\theta(\bar y),u(t;\bar y))$. 

In particular, if $k_i=i-1$, $i=1,\ldots,n$, then the condition (\ref{cond2}) is satisfied automatically. Moreover, in this case the moment problem (\ref{mom_gaps}) has no gaps, hence, it can be effectively and completely solved by the method described in \cite{KS}; see also \cite{KSI} for additional comments and examples. 

For the power Markov moment min-problem with gaps (\ref{mom_gaps}) of the general form, a deep study was obtained in \cite{KS_gap}. One particular case of even gaps was treated in \cite{KB}.  

\subsection{Conditions of linearizability for non-autonomous systems}\label{subsec2_2}

In \cite{K}, linearizability conditions for nonlinear non-autonomous control systems were given; further analysis can be found in \cite{K1}, \cite{K2}. In this subsection we formulate a direct corollary of these results related to a local statement of the problem. 

First, let us notice that if a system of the form (\ref{system_nonlin_gen}) is locally linearizable, then it is of the affine form 
\begin{equation}\label{system_nonlin_aff}
	\dot x=a(t,x)+b(t,x)u, \ \ a(t,0)\equiv0,
\end{equation}
where the condition $a(t,0)\equiv0$ means that the origin is an equilibrium of the system. Denote by $\RR$ the following operator which acts on a vector function $c(t,x)$ by the rule
\begin{displaymath}
	\RR c(t,x)=c_t(t,x)+c_x(t,x)a(t,x)-a_x(t,x)c(t,x),
\end{displaymath} 
where sub-indices $t$ and $x$ denote the derivatives w.r.t.\ $t$ and $x$ respectively. Introduce the following matrix  
\begin{displaymath}
	R(t,x)=(b(t,x),\RR b(t,x),\ldots,\RR^{n-1}b(t,x)).
\end{displaymath}
By $[\cdot,\cdot]$ we denote the Lie bracket, $[c(t,x),d(t,x)]=d_x(t,x)c(t,x)-c_x(t,x)d(t,x)$. Also we use the notation  $k^{\underline{j}}$ for the \textit{falling factorial},  
\begin{displaymath}
	k^{\underline{j}}=k(k-1)\cdots(k-j+1), \ j\ge1, \ \ \ k^{\underline{0}}=1.
\end{displaymath} 

\begin{thm}\label{th2}
	Consider an affine non-autonomous control system of the form (\ref{system_nonlin_aff}), where $a(t,x)\in C^2((-\delta,\delta)\times U(0))$, $b(t,x)\in C^1((-\delta,\delta)\times U(0))$. This system is locally linearizable if and only is all vector functions $\RR^i b(t,x)$ for $1\le i\le n$ exist and belong to the class $C^1((-\delta,\delta)\times U(0))$ and the following conditions are satisfied,
	\begin{enumerate}
		\item $[\RR^ib(t,x),\RR^jb(t,x)]=0$ for $0\le i<j\le n-1$, $(t,x)\in (-\delta,\delta)\times U(0)$;
		\item $\rank R(t,x)=n$ for $t\in(-\delta,0)\cup(0,\delta)$ and $x\in U(0)$;
		\item the vector function $R^{-1}(t,x)\RR^nb(t,x)$ depends only on $t$, i.e., 
		\begin{displaymath}
			R^{-1}(t,x)\RR^nb(t,x)=\gamma(t);
		\end{displaymath}
		\item components of $\gamma(t)$ are analytic or meromorphic functions in a neighborhood of the point $t=0$ with a pole at $t=0$ such that
		\begin{displaymath}\gamma_i(t)=\sum_{j=-n+i-1}^\infty \gamma_{i,j}t^j, \ \ i=1,\ldots,n,
		\end{displaymath}
			the indicial equation 
			\begin{displaymath}k^{\underline{n}}-\sum_{s=1}^nk^{\underline{n-s}}\gamma_{n-s+1,-s}=0
			\end{displaymath} 
			has $n$ integer nonnegative roots $0\le k_1<\cdots<k_n$	and
			\begin{equation}\label{cond_4}
				\rank \begin{pmatrix}
					V_{k_1+1,k_1}&V_{k_1+1,k_1+1}&0&\cdots&0\\
					\cdots&\cdots&\cdots&\cdots&\cdots\\
					V_{k_n,k_1}&V_{k_n,k_1+1}&V_{k_n,k_1+2}&\cdots&V_{k_n,k_n-1}
				\end{pmatrix}=k_n-k_1-n+1,
			\end{equation}
			where
			\begin{displaymath}
				\begin{array}{l}
					\displaystyle V_{k,k}=k^{\underline{n}}-\sum_{s=1}^nk^{\underline{n-s}}\gamma_{n-s+1,-s},\\
					\displaystyle V_{k,j}=-\sum_{s=1}^n	j^{\underline{n-s}}\gamma_{n-s+1,k-j-s}, \ \ j\le k-1.			
				\end{array}
			\end{displaymath}
	\end{enumerate}
\end{thm}

\begin{rem}
	Conditions 1 and 2 of Theorem~\ref{th3} are analogous to linearizability conditions for autonomous systems as well as the requirements $\RR^kb(t,x)\in C^1((-\delta,\delta)\times U(0))$, $k=0,\ldots,n$ \cite{SCL_2005}. Conditions 3 and 4 are specific for non-autonomous case \cite{K}, \cite{K2}. Condition $a(t,x)\in C^2((-\delta,\delta)\times U(0))$ is of technical character, see \cite{K1} for a detailed discussion.
\end{rem}

\begin{rem}
	If the system (\ref{system_nonlin_aff}) is linearizable, its linear representation (\ref{system_lin_y}) is not unique. It is convenient to choose it in a \textit{driftless form} 	 
	\begin{equation}\label{dr_form}
		\dot z=g(t)u,
	\end{equation}
which can be considered as a canonical form for linear systems suitable both for autonomous and non-autonomous cases. Components of $g(t)$ can be found as $n$ linearly independent real analytic solutions of the differential equation
	\begin{equation}\label{dif_eq}
		w^{(n)}=\sum_{k=1}^n\gamma_{k}(t)w^{(k-1)},
	\end{equation}
where $w^{(j)}$ denotes the $j$-th derivative in $t$. If the analytic solving of the differential equation (\ref{dif_eq}) is impossible, one can find sufficiently many coefficients of the Taylor series for a solution using the recurrent formula 
 \begin{equation}\label{rec_form}
 	w_k=-\frac{1}{V_{k,k}}\sum_{j=0}^{k-1}V_{k,j}w_j, \ \ k\ge0, \ k\ne k_i, \ i=1,\ldots,n,
 \end{equation}
where $w_{k_1},\ldots,w_{k_n}$ are arbitrary. 

It is convenient to choose $g_i(t)$ such that $g_i(t)=-t^{k_i}+o(t^{k_n})$; in this case $L=I$. When using (\ref{rec_form}), one should choose $w_{k_i}=-1$ and  $w_{k_j}=0$ for $j\ne i$.
 \end{rem}

\begin{rem} A change of variables $F(t,x)$ satisfies the following partial differential equations
	\begin{displaymath}
			F_t(t,x)+F_x(t,x)a(t,x)=0, \ F_x(t,x)\RR^k b(t,x)=g^{(k)}(t), \ \ k\ge0.
	\end{displaymath} 
It is more convenient to find it as a solution of the system  
\begin{equation}\label{dif_eq_F}
		F_x(t,x)=G(t)R^{-1}(t,x), \ F_t(t,x)=-F_x(t,x)a(t,x),
\end{equation}
where $G(t)=(g(t),\dot g(t),\ldots,g^{(n-1)}(t))$; see also Remark~\ref{rem1} below.
\end{rem}

\section{Main result}\label{sec3}

Now we combine the theorems formulated in the previous section and present our main result.  

\begin{thm}\label{th3}
	Suppose that the system (\ref{system_nonlin_aff}) satisfies the conditions of Theorem~\ref{th2} and, additionally, 
	\begin{equation}\label{cond_5}
	V_{\ell,k_i}=0 \ \ \mbox{ for } \ \ \ell = k_i+1,\ldots,k_n, \ \ i=1,\ldots,n-1.	
	\end{equation} 
Then there exist $\delta'>0$, a neighborhood $U'(0)$ of the origin and a local change of variables $z=F(t,x)\in C^2((-\delta',\delta')\times U'(0))$ such that for any $x^0\in U'(0)$ the solution $(\widehat\theta(x^0),\widehat u(t;x^0))$ of the time-optimal control problem 
\begin{displaymath}
	\dot x=a(t,x)+b(t,x)u, \ x(0)=x^0, \ x(\theta)=0, \ |u(t)|\le1, \ \theta\to\min
\end{displaymath} 
can be found by the method of successive approximations as (\ref{s_app1}), where 
\begin{equation}\label{s_app4}
	y^0=L^{-1}F(0,x^0), \ \  y^{r+1}=L^{-1}\left(F(0,x^0)+\int_0^{\theta(y_r)}g(t)u(t;y^r)dt\right)+y^r, \ \ r\ge0.
\end{equation}
Here components of $g(t)$ are $n$ linearly independent analytic solutions of the differential equation (\ref{dif_eq}), $L=(-\frac{1}{k_1!}g^{(k_1)}(0),\ldots,-\frac{1}{k_n!}g^{(k_n)}(0))$, the vector function $F(t,x)$ satisfies the system of differential equations (\ref{dif_eq_F}) and $F(0,0)=0$. 
\end{thm}

\begin{proof}
	We notice that the condition (\ref{cond_5}) obviously implies (\ref{cond_4}). 
	
	To prove the theorem, it is sufficient to show that (\ref{cond_5}) implies (\ref{cond2}). However, it easily follows from the recurrent formula (\ref{rec_form}).
\end{proof}

\begin{rem}\label{rem1}
	To apply Theorem~\ref{th3}, it is not necessary to solve the system (\ref{dif_eq_F}). In fact, we only need to find $F(0,x^0)$, so we can proceed as follows. Denote $M(x)=G(t)R^{-1}(t,x)|_{t=0}$ and,  for any $k=1,\ldots,n$, consider the equations 
\begin{displaymath}
	\frac{\partial F_k(0,x)}{\partial x_s}=M_{ks}(x), \ \ s=1,\ldots,n.
\end{displaymath}	
Successively for $s=1,\ldots,n$, solve (at least, numerically) the Cauchy problem for a single ordinary differential equation 
\begin{displaymath}
	z'(\tau)=M_{ks}(0,x_1^0,\ldots,x^0_{s-1},\tau,0,\ldots,0), \ \ z(0)=F_k(0,x_1^0,\ldots,x^0_{s-1},0,\ldots,0),
\end{displaymath}
then $F_k(0,x_1^0,\ldots,x^0_{s-1},x^0_s,0,\ldots,0)=z(x_s^0)$. We find $F_k(0,x^0)$ after $n$ such steps. 
\end{rem}

\section{Example}\label{sec4}
As an illustrative example, we consider the following system
\begin{equation}\label{Ex1_nonlin}
	\begin{array}{l}
		\dot x_1=u,\\
		\dot x_2=(t-{\textstyle \frac13} t^4-2x_1x_3-(2t^2+t^3+{\textstyle \frac15}t^5)x_1^2)u, \\
		\dot x_3=(t^3+{\textstyle \frac15}t^3-t^2)u-2tx_1.
	\end{array}
\end{equation}
Let us check the conditions of Theorem~\ref{th3}. We have
\begin{displaymath}
	a(t,x)=\begin{pmatrix}
		0\\0\\-2tx_1
	\end{pmatrix}, \ \ b(t,x)=\begin{pmatrix}
	1\\ t-{\textstyle \frac13} t^4-2x_1x_3-(2t^2+t^3+{\textstyle \frac15}t^5)x_1^2\\ t^3+{\textstyle \frac15}t^5-t^2
\end{pmatrix},
\end{displaymath}
\begin{displaymath}
	\RR b(t,x)=\begin{pmatrix}
		0\\1-{\textstyle \frac43}t^3-(3t^2+t^4)x_1^2\\3t^2+t^4
	\end{pmatrix}, \ \ \RR^2b(t,x)=\begin{pmatrix}
		0\\ -4t^2-(6t+4t^3)x_1^2\\ 6t+4t^3
	\end{pmatrix},
\end{displaymath}
\begin{displaymath}
	R(t,x) = (b(t,x),\RR b(t,x),\RR^2b(t,x)), \ \ \ \RR^3b(t,x)=\begin{pmatrix}
		0\\-8t-(6+12t^2)x_1^2\\ 6+12t^2
	\end{pmatrix}.
\end{displaymath}
First, we notice that $\det R(t,x)=6t+4t^3+4t^4-{\textstyle\frac43}t^6$, hence, $R(t,x)$ is nonsingular at points $(t,x)$ from some neighborhood of the origin such that $t\ne0$. Then, we find that $[b(t,x),\RR b(t,x)]=[b(t,x),\RR^2 b(t,x)]=[\RR b(t,x),\RR^2 b(t,x)]=0$, therefore, conditions 1 and 2 of Theorem~\ref{th2} hold. 

Further, 
\begin{displaymath}
	R^{-1}(t,x)\RR^3b(t,x)=\begin{pmatrix}
		0\\[2pt]
		\frac{12t(2t^2-3)}{9+6t^2+6t^3-2t^5}\\[6pt]
		\frac{3(3+6t^2+8t^3-4t^5)}{t(9+6t^2+6t^3-2t^5)}
	\end{pmatrix}=\begin{pmatrix}
		\gamma_1(t)\\\gamma_2(t)\\ \gamma_3(t)
	\end{pmatrix},
\end{displaymath}
where $\gamma_i(t)$ are analytic or meromorphic and $\gamma_1(t)=0$, $\gamma_2(t)=-4t+O(t^3)$, $\gamma_3(t)=\frac{1}{t}+O(t)$. Hence, the indicial equation is  
\begin{displaymath}
	k^{\underline{3}}-k^{\underline{2}}\gamma_{3,-1}-k^{\underline{1}}\gamma_{2,-2}-k^{\underline{0}}\gamma_{1,-3}=0;
\end{displaymath}	
since $\gamma_{1,-3}=0$, $\gamma_{2,-2}=0$, $\gamma_{3,-1}=1$, it takes the form
\begin{displaymath}
		 k(k-1)(k-2)-k(k-1)=0
\end{displaymath}
and has the roots $k_1=0$, $k_2=1$, $k_3=3$. Now,
\begin{displaymath}
	V_{1,0}=-\gamma_{1,-2}=0, \ \ V_{2,0}=-\gamma_{1,-1}=0, \ \ V_{3,0}=-\gamma_{1,0}=0,
\end{displaymath}
\begin{displaymath}
	V_{2,1}=-\gamma_{2,-1}-\gamma_{1,-2}=0, \ V_{3,1}=-\gamma_{2,0}-\gamma_{1,-1}=0,
\end{displaymath}
hence, conditions 3 and 4 of Theorem~\ref{th2} and the condition (\ref{cond_5}) are satisfied. Hence, the system (\ref{Ex1_nonlin}) is locally linearizable and the time-optimal control problem for this system can be solved by the method of successive approximations. 

We find a system after linearization in the driftless form (\ref{dr_form}), where the components of $g(t)$ are solutions of the following differential equation
\begin{displaymath}
	w'''=\gamma_1(t)w+\gamma_2(t)w'+\gamma_3(t)w''.
\end{displaymath} 
In our case, we have
\begin{displaymath}
	t(9+6t^2+6t^3-2t^5)w'''=12t^2(2t^2-3)w'+3(3+6t^2+8t^3-4t^5)w'',
\end{displaymath}
and it is easy to check that $g_1(t)=-1$, $g_2(t)=-t+\frac13t^4$ and $g_3(t)=-t^3-\frac15t^5$ are its three linearly independent solutions. Hence, $L=I$ and in the new coordinates the system takes the driftless form 
\begin{displaymath}
	\begin{array}{l}
		\dot z_1=-u, \\ \dot z_2=-(t-{\textstyle \frac13} t^4)u, \\ \dot z_3=-(t^3+{\textstyle \frac15} t^5)u.
	\end{array}
\end{displaymath}
In this case, the power Markov moment min-problem (\ref{mom_gaps}) is of the form
\begin{equation}\label{mom_pr_Ex}
	y_1=\int_0^\theta u(t)dt, \  y_2=\int_0^\theta t\,u(t)dt,  \ y_3=\int_0^\theta t^3u(t)dt, \ \ |u(t)|\le1, \ \theta\to\min.
\end{equation}
Its solution can be found directly. In fact, the optimal control is unique and equals $\pm1$ and has no more than two switchings. 
Since the set of points for which it has \textit{less} than two switchings is of zero measure, for numerical float-point calculation it is sufficient to assume that there are \textit{exactly} two switchings; denote them by $t_1$ and $t_2$, and let $\theta$ be the optimal time. Then
\begin{displaymath}
	\begin{array}{l}
		\pm y_1=-2t_1+2t_2-\theta,\\
		\pm y_2=-t_1^2+t_2^2-\frac12\theta^2,\\
		\pm y_3=-\frac12t_1^4+\frac12t_2^4-\frac14\theta^4,
	\end{array}
\end{displaymath}  
where the upper (resp., lower) sign means that $u(t)$ equals $+1$ (resp, $-1$) on the first and the third intervals of constancy. Let us denote $c_1^\pm=\frac12(\pm y_1+\theta)$, $c_2^\pm=\pm y_2+\frac12\theta^2$, $c_3=2(\pm y_3+\frac14\theta^4)$, then 
\begin{displaymath}
	t_2-t_1=c_1^\pm, \ \ t_2^2-t_1^2=c_2^\pm, \ \ t_2^4-t_1^4=c_3^\pm.
\end{displaymath}
Excluding $t_1$ and $t_2$, we get two equations w.r.t. $\theta$ 
\begin{displaymath}
	2c_3^\pm (c_1^\pm)^2=(c_2^\pm)^3+c_2^\pm(c_1^\pm)^4;
\end{displaymath}
they are polynomial equations in $\theta$ of degree 6. There exists a unique root $\theta$ of one of these equations such that $0< t_1< t_2<\theta$. Hence, it it easy to find the solution of (\ref{mom_pr_Ex}) numerically for any $y$.

Finally, let us find $F(0,x^0)$. We have
\begin{displaymath}
	M(x)=G(t)R^{-1}(t,x)=\begin{pmatrix}
		-1&0&0\\
		-2x_1x_3&-1&-x_1^2\\
		0&0&-1
	\end{pmatrix}.
\end{displaymath} 
Hence, $\frac{\partial F_1(0,x)}{\partial x_1}=-1$, $\frac{\partial F_1(0,x)}{\partial x_2}=\frac{\partial F_1(0,x)}{\partial x_3}=0$, therefore, $F_1(0,x^0)=-x_1^0$. Analogously, $F_3(0,x^0)=-x^0_3$. For $F_2(0,x)$, we have  $\frac{\partial F_2(0,x)}{\partial x_1}=-2x_1x_3$, $\frac{\partial F_1(0,x)}{\partial x_2}=-1$, $\frac{\partial F_1(0,x)}{\partial x_3}=-x_1^2$. In this particular case the solution is obvious. However, we demonstrate the method of finding $F_2(0,x^0)$ described in Remark~\ref{rem1}.

Since $F(0,0,0,0)=0$, solving the Cauchy problem $z'(\tau)=F_2(0,0,0,0)=0$, $z(0)=0$ we get $z(x_1^0)=F_2(0,x^0_1,0,0)=0$. Then, considering the Cauchy problem $z'(\tau)=-1$, $z(0)=0$, we get $z(x_2^0)=F_2(0,x_1^0,x_2^0,0)=-x_2^0$. Finally, solving the Cauchy problem $z'(\tau)=-(x_1^0)^2$, $z(0)=-x_2^0$, we get $z(x_3^0)=F_2(0,x^0)=-x_2^0-(x_1^0)^2x_3^0$. As a result, $F(0,x^0)=(-x_1^0,-x_2^0-(x_1^0)^2x_3^0,-x_3^0)$. 

Suppose we solve the time-optimal control problem for the system (\ref{Ex1_nonlin}) from the point $x^0=(-0.4,-0.2,0.1)$, then $y^0=F(0,x^0)=(0.4,0.184,-0.1)$. Using the method of successive approximations we get  $\lim\limits_{r\to\infty} y^r\approx(0.4,0.1457,-0.0714)$ and $t_1\approx0.1251$, $t_2\approx0.8740$, $\theta\approx 1.0978$. After 45 iterations one achieves $\|y^{r+1}-y^r\|<10^{-8}$; the trajectory components are shown in Fig.~\ref{fig1}~(a).

\begin{figure}
	\begin{minipage}{\textwidth}
		\includegraphics[width=0.49\textwidth]{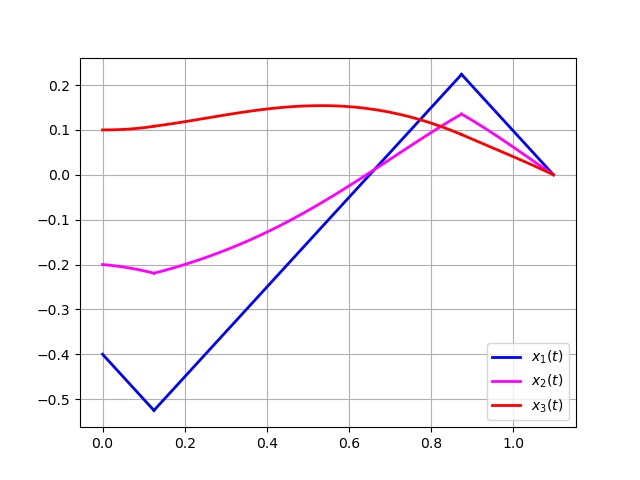}
		\hfill
		\includegraphics[width=0.49\textwidth]{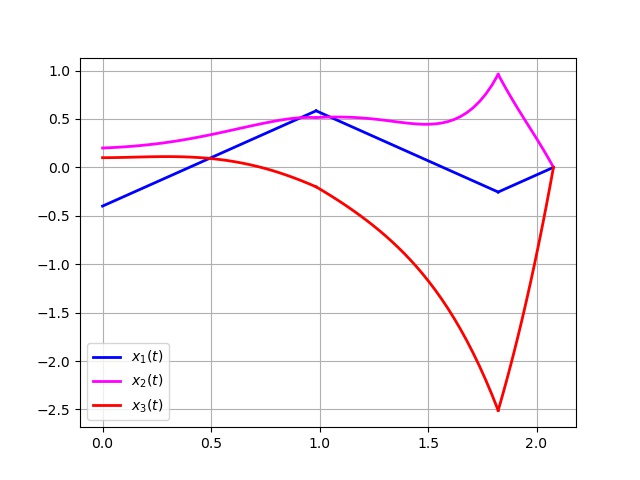}\\
		\hspace*{0.2\textwidth} (a) \hspace*{0.5\textwidth} (b) 
	\end{minipage}
	\caption{Components of the optimal trajectory: (a)~${x^0=(-0.4,-0.2,0.1)}$; (b)~${x^0=(-0.4,0.2,0.1)}$.
	}\label{fig1}
\end{figure}

If the starting point for the initial system is $x^0=(-0.4,0.2,0.1)$, we get $y^0=(0.4,-0.216,-0.1)$ and the method of successive approximations diverges. However, one can apply the following modification: instead of (\ref{s_app4}), use the formula
\begin{displaymath}
	y^0=L^{-1}F(0,x^0), \ \  y^{r+1}=cL^{-1}\left(F(0,x^0)+\int_0^{\theta(y_r)}g(t)u(t;y^r)dt\right)+y^r, \ \ r\ge0,
\end{displaymath}
where $c\in(0,1)$ (recall that in our example $L=I$). One can show that the mapping leading to this recursive formula is also a contraction. Though the contraction constant is greater, a domain where the method converges can be wider. So, in the previous example, if $c=\frac14$, then $y^r$ converges; after 120 iterations one achieves $\|y^{r+1}-y^r\|<10^{-8}$. We obtain $t_1\approx1.8232$, $t_2\approx2.0779$, $\theta\approx 0.9843$; the trajectory components are shown in Fig.~\ref{fig1}~(b).

\end{document}